\documentclass[12pt,a4paper]{amsart}
\usepackage{amsmath,amsthm,amsfonts,amssymb,mathrsfs}

\usepackage{graphicx}

\usepackage{amssymb}

\usepackage{color}

\begin{document}

\newtheorem{thm}{Theorem}
\newtheorem{lem}[thm]{Lemma}
\newtheorem{cor}[thm]{Corollary}
\newtheorem{prop}[thm]{Proposition}

\theoremstyle{remark}
\newtheorem{rem}[thm]{Remark}


\newcommand{\RR}{\mathbb R}
\newcommand{\R}{\mathbb R}
\newcommand{\be}{\begin{equation}} 
\newcommand{\ee}{\end{equation}}
\newcommand{\bea}{\begin{eqnarray}} 
\newcommand{\eea}{\end{eqnarray}}
\newcommand{\bean}{\begin{eqnarray*}} 
\newcommand{\eean}{\end{eqnarray*}}
\newcommand{\rf}[1]{(\ref {#1})}
\newcommand{\un}{{\rm 1\!\!I}}
\newcommand{\dx}{\,{\rm d}x}
\newcommand{\dy}{\,{\rm d}y}
\newcommand{\dr}{\,{\rm d}r} 
\newcommand{\dz}{\,{\rm d}z}
\newcommand{\dt}{\,{\rm d}t}
\newcommand{\ds}{\,{\rm d}s}
\newcommand{\dsi}{\, {\rm d}\sigma}
\newcommand{\dxi}{\,{\rm d}\xi} 
\newcommand{\la}{\langle}
\newcommand{\ra}{\rangle} 
\def\xn{|\!|\!|}
\def\e{\varepsilon}
\def\p{\partial}
\def\div{\nabla\cdot}
\def\f{\varphi}
\def\r{\varrho}


\title[Parabolic-parabolic Keller-Segel system]{Large global-in-time solutions\\ of the parabolic-parabolic \\ Keller-Segel system on the plane}

\author{Piotr Biler}
\address[P. Biler]{
 Instytut Matematyczny, Uniwersytet Wroc\l awski,
 pl. Grunwaldzki 2/4, 50-384 Wroc\-\l aw, Poland}
\email{piotr.biler@math.uni.wroc.pl}

\author{ Ignacio Guerra}
\address[I. Guerra]{
Departamento de Matem\'atica y de la Ciencia de Computaci\'on,
Universidad de Santiago de Chile, Chile}
\email{ignacio.guerra@usach.cl}

\author{Grzegorz Karch}
\address[G. Karch]{
 Instytut Matematyczny, Uniwersytet Wroc\l awski,
 pl. Grunwaldzki 2/4, 50-384 Wroc\-\l aw, Poland}
\email{grzegorz.karch@math.uni.wroc.pl}
\urladdr {http://www.math.uni.wroc.pl/~karch}


\date{\today}
\maketitle

\begin{abstract} 
\noindent
As it is well known, the parabolic-elliptic Keller-Segel system of chemotaxis  on the plane has global-in-time regular nonnegative solutions with total mass below the critical value $8\pi$. Solutions with mass above $8\pi$ blow up in a finite time. 
We show that the case of the parabolic-parabolic Keller-Segel is different: each mass may lead to a global-in-time-solution, even if the initial data is a  finite signed measure. 
 These solutions need not be unique, even if we limit ourselves to nonnegative solutions. 

\bigskip

{\noindent{\bf Key words and phrases:}{ chemotaxis, parabolic-parabolic  Keller-Segel model, large global-in-time  solutions}}

{\noindent {\bf 2000 Mathematics Subject Classification:}  35Q92, 35K40.} 
\end{abstract}


\section{Introduction} 
We consider in this paper the simplest doubly parabolic version of the Keller-Segel model of chemotaxis 
\bea 
u_t &=&\nabla\cdot(\nabla u-u\nabla v),\ \ \ x\in\mathbb R^2,\ t>0,\label{pop}\\
\tau v_t &=& \Delta v-\gamma v+u,\ \ \ \ \ \ \ x\in\mathbb R^2,\ t>0,\label{ch}
\eea
where $u$ denotes the density of a population and $v$ -- the density of a chemical, called chemoattractant, which is secreted by the microorganisms and makes them to attract themselves. The equations are supplemented with the initial data 
\be
u(\cdot,0)=u_0,\ \ \ v(\cdot,0)=0,\label{ini}
\ee
which we suppose to be a finite Radon measure $u_0\in{\mathcal M}(\mathbb R^2)$. 
We  choose  $v(x,0)=0$
 for simplicity, however, the analogous computations could be done with every sufficiently regular $v(x,0)$. 
Here, the constant parameter  $\tau>0$ is related to the diffusion rate of the chemical, and usually in applications $\tau$ is small since the chemoattractant diffuses much faster than the population. We are interested, however,  in arbitrary positive values of $\tau$. The coefficient $\gamma\ge 0$ is the consumption rate of the chemical.

It is well known that mass $M=\int u(x,t) \dx$ is conserved for solutions of \rf{pop}--\rf{ini}. Moreover, positivity of the initial data is preserved during the evolution: $u(x,t)\ge 0$, $v(x,t)\ge 0$, see the references listed below. However, we will consider in the sequel solutions of arbitrary sign, and thus, we will not use that positivity-preserving property. 

The limit case $\tau=0$ is called the parabolic-elliptic Keller-Segel model, and  has been much more studied than the doubly parabolic one with $\tau>0$. For the relations between those two systems as $\tau\searrow 0$, see, e.g., \cite{R}, \cite{BB} and \cite{Lem}. 

Let us first recall that in the parabolic-elliptic case ($\tau=0$) the existence of solutions of \rf{pop}--\rf{ch} has been studied in, e.g.,  \cite{B-SM}, \cite{SS}, \cite{KS-JEE}, \cite{BM}. In particular, for $M>8\pi$,  the corresponding 
 nonnegative solution cannot be global in time. Moreover, 
 measures as initial conditions with atoms bigger than $8\pi$ were obstructions even for the local-in-time existence of solutions. Self-similar asymptotics for $M<8\pi$ has been proved in  \cite{BDP}, and asymptotics at the critical value $M=8\pi$ has been considered in \cite{BKLN} (the radial case) and \cite{BCM} (the general case). 
Continuation past blowup time was the topic of \cite{B-BCP} (the radial case) and \cite{DS} (the general case).

In the following theorem, which is the main result of this note, we show that one should not expect any critical mass 
determining the existence of solutions to parabolic-parabolic Keller-Segel model 
\rf{pop}--\rf{ini}
with sufficiently large $\tau>0$.

\begin{thm}
For each $u_0\in {\mathcal M}(\mathbb R^2)$ there exists $\tau(u_0)>0$ such that for all $\tau\ge\tau(u_0)$ the Cauchy problem \rf{pop}--\rf{ini} has a global-in-time mild solution satisfying $u\in {\mathcal C}_{\rm w}([0,\infty);\mathcal M(\mathbb R^2))$. This is a classical solution of the system \rf{pop}--\rf{ch}  for $t>0$, and satisfies  
\be
\sup_{t>0}t^{1-1/p}\|u(t)\|_p<\infty\label{esti}
\ee 
for each $p\in[1,\infty]$. 
\end{thm}

Theorem 1 improves  results for the  parabolic-parabolic Keller-Segel model ($\tau>0$)  in \cite{B-AMSA1}, \cite{CC} and \cite{M}, where the global existence of solutions for $M<8\pi$ on the whole plane has been considered. 
Self-similar solutions with large masses (above $8\pi$) have been constructed in \cite{BCD}. 
Blowup of radially symmetric solutions in balls is a very recent result, \cite{MW}. 
On the other hand, there exist large (unstable) stationary solutions in balls, see \cite[Ch. 6]{B-AMSA1}. 

 All those results show that, unlike the parabolic-elliptic case, there is no threshold value of mass for the local existence of solutions as well as for the global-in-time existence. 

\begin{rem}\label{rem:self}
In general, solutions of the Cauchy problem \rf{pop}--\rf{ini} constructed in Theorem 1  need not be unique. This striking property  is seen when we consider certain radially symmetric nonnegative self-similar solutions for  the system \rf{pop}--\rf{ch} with $\gamma=0$ --- which are of the scaling invariant form 
\be
u(x,t)=\frac{1}{t}U\left(\frac{|x|}{\sqrt{t}}\right),\ \ \ 
v(x,t)=V\left(\frac{|x|}{\sqrt{t}}\right),\label{ss}
\ee
for some functions $U$, $V$ of one variable. 
They   have been constructed using ODE methods in \cite{BCD} and  they correspond to the initial data $u_0=M\delta_0$, $v_0=0$, with the Dirac measure $\delta_0$. 
In particular, for $0<\tau\le \frac12$, they exist exactly in the range $M\in[0,8\pi)$. 
However, for $\tau>\tau^\ast$ ($\approx 0.64  $) there exist self-similar solutions with $M\in[0,M_\tau]$ with at least two solutions for each $M\in(8\pi,M_\tau)$,  $M_\tau>8\pi$, and even $\lim_{\tau\to\infty}M_\tau =\infty$ since it follows  from \cite[Th. 4]{BCD} that $M_\tau\ge  \frac{4\pi}{{\rm e}}\frac{\tau-1}{\log \tau}$.  
\end{rem}
\medskip

Finally, let us formulate an important consequence of out result. For arbitrary $M>0$ and for all $\tau\geq \tau(M)$, Theorem 1 provides us with a  solution of the Cauchy problem \rf{pop}--\rf{ini} with $u_0=M\delta_0$ and $\gamma=0$.
By a standard argument, one may show that it has  the self-similar form \rf{ss}.  
On the other hand, by Remark \ref{rem:self}, there exists another self-similar solution with the same value of $M>8\pi$. This is, to the best of our knowledge, {\it the first nontrivial example of nonuniqueness} of mild solutions to a chemotaxis system with measures as initial conditions.

\section{Preliminaries} 
Let us denote by ${\rm e}^{t\Delta}$ the heat semigroup on $\mathbb R^2$ acting as the convolution with the Gauss--Weierstrass kernel 
$G(x,t)=(4\pi t)^{-1}\exp\left(-|x|^2/(4t)\right).$ 
The standard estimates for the regularization effect  and the decay rates for solutions of the heat equation have the following form  
\be
\|{\rm e}^{t\Delta}z\|_q\le Ct^{1/q-1/r}\|z\|_r\label{lin1}
\ee
and 
\be 
\|\nabla{\rm e}^{t\Delta}z\|_q\le Ct^{-1/2+1/q-1/r}\|z\|_r\label{lin2}
\ee
 for all $1\le r\le q\le\infty$. Here, $\|\cdot\|_q$ denotes the usual $L^q(\mathbb R^2)$ norm, and $C$'s are generic constants independent of $t$, $u$, $z$, ... ,  which may, however, vary from line to line. ${\mathcal M}(\mathbb R^2)$ denotes the Banach space of finite Radon measures on $\mathbb R^2$ with the usual total variation norm, and the weak convergence tested with all continuous compactly supported functions $\varphi\in C_0(\R^2)$. 

An immediate consequence of \rf{lin1} is the bound
$$ 
\sup_{t>0}t^{1-1/p}\|{\rm e}^{t\Delta}u_0\|_p\le C\|u_0\|_1.$$ 
Analogously, it is known that the inequality 
\be 
 \sup_{t>0}t^{1-1/p}\|{\rm e}^{t\Delta}\mu\|_p\le C\|\mu\|_{\mathcal M(\mathbb R^2)}\label{lin4}
\ee
holds true for all $\mu\in \mathcal M(\mathbb R^2)$.

The {\it mild} formulation of   system \rf{pop}--\rf{ch}, together with  initial conditions \rf{ini} is the integral equation (a. k. a. the Duhamel formula) 
\be
u(t)={\rm e}^{t\Delta}u_0+B(u,u)(t),\label{D}
\ee
where the quadratic form $B$ is defined by 
\be
B(u,z)(t)=-\int_0^t\left(\nabla{\rm e}^{(t-s)\Delta}\right)\cdot\left(u(s)\, Lz(s)\right)\ds, \label{B}
\ee 
with the solution operator of \rf{ch} 
\be 
Lz(t)=\tau^{-1}\int_0^t\left(\nabla{\rm e}^{\tau^{-1}(t-s)(\Delta-\gamma)}\right) z(s)\ds. \label{L}
\ee
The existence of solutions of the quadratic equation \rf{D} is established by the usual 
approach using the contraction argument in a suitable functional space of vector-valued functions. In our case, that space is denoted by 
$$ {\mathcal E}_p=\{u\in L^\infty_{\rm loc}((0,\infty); L^p(\mathbb R^2)):\ \ \sup_{t>0}t^{1-1/p}\|u(t)\|_p<\infty\},$$ 
and  the norm $\xn\, .\, \xn_p$ in ${\mathcal E}_p$ is defined as 
\be
\xn u\xn_p\equiv \sup_{t>0}t^{1-1/p}\|u(t)\|_p<\infty.\label{norm}
\ee 
  
Then, we will show that actually $u\in {\mathcal X}$ where 
$$\mathcal X={\mathcal C}_{\rm w}([0,\infty); \mathcal M(\mathbb R^2))\cap {\mathcal E}_p.$$

\begin{rem}
Note that solutions of the equation $u=y_0+B(u,u)$ (more general than \rf{D}) in a Banach space $({\mathcal Y},\|.\|_{\mathcal Y})$  provided by that contraction argument (or, equivalently, by the Picard iteration scheme) are locally unique, but they need not be unique in general as Fig. 1 shows for ${\mathcal Y}=\mathbb R$
and for the quadratic equation 
 $u=y_0+\eta u^2$ with  a fixed $\eta>0$ and 
$|y_0|<\frac{1}{4\eta}$. 
\end{rem}

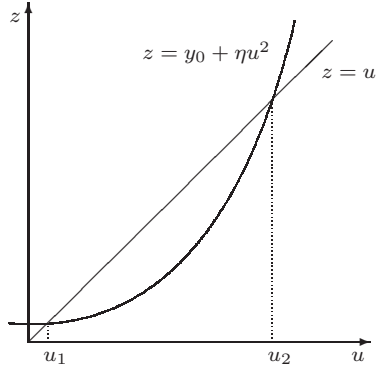
\begin{figure}
  \setlength{\unitlength}{0.5mm}

{\tiny

\begin{picture}(100,100)

\put(10,10){\vector(0, 1){90}}
\put(10,10){\vector(1, 0){90}}

\put(95,5){$u$}
\put(5,95){$z$}
\put(87,80){$z=u$}
\put(40,85){$z=y_0+\eta u^2$}

\put(14,5){$u_1$}
\put(73,5){$u_2$}

\multiput(74,10)(0,1){65}{\line(0,1){0.1}}
\multiput(15,10)(0,1){5}{\line(0,1){0.1}}

\put(10,10){\line(1, 1){80}}

\qbezier(5,15)(60,10)(80,95)

  \end{picture}
  }
  \caption{Two solutions $u_1$ and $u_2$ of the quadratic equation $u=y_0+\eta u^2$.}
\end{figure}


\section{Proof of the main result	} 

The proof of Theorem 1 is split into two parts. In the first, solutions of \rf{D} are constructed in $\mathcal E_p$ with a fixed $p\in\left(\frac43,2\right)$. Then, they are shown to attain the initial data in the weak sense, i.e. they belong to $\mathcal X={\mathcal C}_{\rm w}([0,\infty); \mathcal M(\mathbb R^2))\cap {\mathcal E}_p.$  

The first part of the proof is based on two lemmata. 

\begin{lem}\label{fix}
If $\xn B(u,z)\xn_p\le \eta\xn u\xn_p\, \xn z \xn_p$ and 
$\xn {\rm e}^{t\Delta}u_0\xn_p \le R<\frac{1}{4\eta}$, then equation \rf{D} has a solution which is  unique in the ball of radius $2R$ in the space  $\mathcal E_p$. 
Moreover,  these solutions depend continuously on the initial data, i.e. $\xn u-\tilde u\xn_p\le C\xn {\rm e}^{t\Delta}(u_0-\tilde u_0)\xn_p$. 
\end{lem}

\begin{proof}
 This is a standard reasoning based on the Banach contraction theorem applied to the operator ${\mathcal E}_p\ni u\mapsto {\rm e}^{t\Delta}u_0+B(u,u)(t)$ in the ball of radius $2R$ in the space $\mathcal E_p$. 
\end{proof}

\begin{lem}\label{form}
Let $p\in\left(\frac43,2\right)$.  
The bilinear form $B$ is bounded from ${\mathcal E}_p\times{\mathcal E}_p$ into ${\mathcal E}_p$: 
$$\xn B(u,z)\xn_p\le \eta\xn u\xn_p\, \xn z \xn_p$$ 
with a constant $\eta=\eta(\tau)$ independent of $u$, $z$ and $\gamma$, such that $\eta(\tau)\to 0$ as $\tau\to\infty$. 
\end{lem}

\begin{proof}
First, we estimate the $L^q$-norm of the linear operator $L$ defined in \rf{L} acting on $z\in L^p(\R^2)$ using the estimates \rf{lin1}--\rf{lin2}. 
Assuming that $1\le p\le q\le\infty$, $p<\infty$,  and $\frac1p-\frac1q<\frac12$ we obtain 
\begin{equation}\label{estL}
\begin{split}
\|Lz(t)\|_q
 &\le C\tau^{-1}\int_0^t(\tau^{-1}(t-s))^{-1/2+1/q-1/p}{\rm e}^{-\gamma\tau^{-1}(t-s)} \|z(s)\|_p\ds \\
 &\le C\tau^{-1/2-1/q+1/p}\int_0^t(t-s)^{-1/2+1/q-1/p}s^{1/p-1}\left( \sup_{0<s\le t}s^{1-1/p}\|z(s)\|_p\right)\ds \\
&\le  C\tau^{-1/2-1/q+1/p}t^{1/q-1/2}  \left(\sup_{0<s\le t}s^{1-1/p}\|z(s)\|_p\right)
\end{split}
\end{equation}
since $-\frac12+\frac1q-\frac1p>-1$ and $\frac1p-1>-1$.  $1/p-1>-1$. 

Next, we may prove  the estimate of  the bilinear form $B$. 
In the following computations, we fix the exponents $p$ and $q$ to have 
\be
\frac43<p\le 2\le p'=\frac{p}{p-1}<q<\frac{2p}{2-p},\label{pq}
\ee 
so that $\left(\frac{p}{p-1},\frac{2p}{2-p}\right)\neq \emptyset$, and the relation 
$\frac1r=\frac1p+\frac1q<1$ defines the exponent $r\in(1,p)$. 
Consequently, we have 
$$-\frac12+\frac1p-\frac1r=-\frac12-\frac1q>-1\ \ \ {\rm and} \ \ \  \frac1q+\frac1p-\frac32>-1,$$  
as well as $\frac1p-\frac1q<\frac12$. 
Thus, using the inequality \rf{estL}, we have the following estimate of the bilinear form $B$   
\begin{equation}\label{Bp}
\begin{split}
\|B(u,z)(t)\| _p 
&\le C\int_0^t(t-s)^{-1/2+1/p-1/r}\|u(s)\,Lz(s)\|_r\ds \\ 
&\le C\int_0^t(t-s)^{-1/2+1/p-1/r}\|u(s)\|_p\|Lz(s)\|_q\ds\\ 
&\le  C\tau^{-1/2-1/q+1/p}\int_0^t(t-s)^{-1/2-1/q} s^{1/p-1}\left(\sup_{0<s\le t}s^{1-1/p}\|u(s)\|_p\right)\\
&\qquad  \times \, s^{-1/2+1/q}\left(\sup_{0<s\le t}s^{1-1/p}\|z(s)\|_p\right)\ds, 
\end{split}
\end{equation}
Since $\int_0^t (t-s)^{-1/2-1/q}s^{1/q+1/p-3/2}\ds = C(p,q) t^{1/p-1}$, recalling  definition \rf{norm} of the norm $\xn \, .\,  \xn_p$, we obtain the inequality 
$$
\xn B(u,z)\xn_p\le \eta \xn u\xn_p \xn z\xn_p,
$$
where $\eta = C\tau^{-1/2-1/q+1/p}$ with the exponent  $-\frac12-\frac1q+\frac1p<0$ 
by the last inequality in \rf{pq}. 
\end{proof}

\begin{proof}[Proof of Theorem 1.] 
Given $u_0\in{\mathcal M}(\mathbb R^2)$, by Lemma \ref{form}  we may choose $\tau(u_0)$ so large to have $\sup_{t>0}t^{1-1/p}\|{\rm e}^{t\Delta}u_0\|_p\le C\|u_0\|_{{\mathcal M}(\mathbb R^2)}<\frac{1}{4\eta(\tau)}$ for all $\tau\ge \tau(u_0)$. Thus, the existence of solutions in $\mathcal E_p$ follows by the application of Lemma \ref{fix}.

Now, using the estimates in Lemma \ref{form} we can interpolate   the estimate $u\in{\mathcal E}_p$ for $p\in \left(\frac43,2\right)$ to get $u\in {\mathcal E}_\sigma$ with $\sigma\in[1,2)$, and next extrapolate this to  for any $\sigma\in (2,\infty]$. Then, a standard reasoning involving parabolic regularization effect for nonhomogeneous heat equation (following e.g. \cite[Th. 4.1]{GMO}) will show that $u(x,t)$ is smooth on $\R^2\times(0,\infty)$. 
 
First, for the interpolation argument, let us compute 
\bea \label{B01}
\|B(u,z)(t)\|_1 &\le& C\int_0^t(t-s)^{-1/2}\|u(s)Lz(s)\|_1\ds \nonumber\\
&\le & C\int_0^t(t-s)^{-1/2}s^{1/p-1}\xn u(s)\xn_p s^{1/p'-1/2}\|Lz(s)\|_{p'}\ds\nonumber\\ & = & C\tau^{-1/2-1/p'+1/p}\xn u\xn_p \xn z\xn_p 
\eea
with the exponent $-\frac12-\frac{1}{p'}+\frac1p<0$,  by \rf{estL} with $q=p'\in(2,3)$ since $p\in\left(\frac43,2\right)$. 
This leads to 
$$
\|u(t)\|_1\le \|{\rm e}^{t\Delta}u_0\|_1+\|B(u,u)(t)\|_1\le C\|u_0\|_{{\mathcal M}(\mathbb R^2)}+C\xn u\xn_p^2,
$$
and $u\in {\mathcal E}_\sigma$ with $\sigma\in(1,2)$ follows from this and $u\in{\mathcal E}_p$ with $p<2$ by the interpolation. Here and below,   the dependence of constants $C$ on $\tau$ is not important, so we skip this.

Then for the extrapolation, given $\sigma\in(2,\infty)$ there exist $p$ and $r$ such that 
$$1<\frac{2\sigma}{2+\sigma}<r<\frac{2\sigma}{1+\sigma}<\frac43<p<2.$$ 
 Applying the proof of Lemma \ref{form} we get for $q$ defined as $\frac1q=\frac1r-\frac1p$ (so that the assumption $\frac1p-\frac1q<\frac12$ is satisfied) 
$$\| Lu(t)\|_q \le Ct^{1/q-1/2}\xn u\xn_p.$$ 
Next, we estimate the bilinear form $B$ in $L^\sigma$ as 
\bea
\|B(u,u)(t)\|_\sigma &\le & \int_0^t (t-s)^{-1/2+1/\sigma-1/r}\|u(s)Lu(s)\|_r \ds\nonumber\\
&\le& \int_0^t(t-s)^{-1/2+1/\sigma-1/r}s^{1/p-1}s^{1/q-1/2}\xn u\xn_p^2\ds\nonumber\\
&\le& Ct^{1/\sigma-1}\xn u\xn_p^2\nonumber
\eea
since $-\frac12+\frac1\sigma-\frac1r>-\frac12+\frac1\sigma-\frac{2+\sigma}{2\sigma}>-1$ and $ \frac1p-1+\frac1q-\frac12=-1+\frac1r-\frac12>-1$ as requested in the proof. 
The above inequalities lead to 
$$\xn u\xn_\sigma\le \xn {\rm e}^{t\Delta} u_0\xn_\sigma+\xn B(u,u)\xn_\sigma\le C\|u_0\|_{{\mathcal M}(\mathbb R^2)}+C\xn u\xn_p^2. $$
The proof in the case $p=\infty$ is completely analogous.

Now, we proceed to the proof that this solution attains the initial data in the weak sense, i.e. this  satisfies $u\in{\mathcal C}_{\rm w}([0,\infty);{\mathcal M}(\mathbb R^2))$. 
For that purpose first we define for a fixed $p\in\left(\frac43,2\right)$  
a subspace of ${\mathcal E}_p$ \ \ 
${\mathcal Y}=L^\infty((0,\infty); {\mathcal M}(\mathbb R^2)) \cap {\mathcal E}_p$ 
endowed with the norm 
$\|u\|_{\mathcal Y}= {\rm ess}\,\sup_{t>0}\|u(t)\|_{{\mathcal M}(\mathbb R^2)}+\xn u\xn_p$. 
We will show that Lemma \ref{fix} applies to equation \rf{D} in the space ${\mathcal Y}$, i.e. the following estimate $\|B(u,v)\|_{\mathcal Y}\le \eta(\tau)\|u\|_{\mathcal Y}\|v\|_{\mathcal Y}$ is  valid with $\eta(\tau)\to 0$ as $\tau\to \infty$. 
The $L^1$ estimate \rf{B01}  together with the bilinear bound for $B$ in the norm of ${\mathcal E}_p$ established in Lemma \ref{form}, and the estimate $\|{\rm e}^{t\Delta}u_0\|_1\le \|u_0\|_1$, show that for fixed $u_0\in {\mathcal M}(\R^2)$ and  $\tau$ sufficiently large  equation \rf{D} has a solution in ${\mathcal Y}$.

Next, to prove that the constructed solution is in fact in the space ${\mathcal X} = {\mathcal C}_{\rm w}([0,\infty); \mathcal M(\mathbb R^2))\cap {\mathcal E}_p \subset {\mathcal Y}$ we will show that $u(t)\rightharpoonup u(s)$ as $t\searrow s\ge 0$ in the sense of the weak convergence of measures. For that purpose, we need only to show that $B(u,u)(t)\rightharpoonup B(u,u)(s)$ as $t\searrow s$; in particular  $B(u,u)(t)\rightharpoonup 0$ as $t\to 0$.
Indeed, the sufficiency of that property is seen from the representation 
$$u(t)-u(s)={\rm e}^{t\Delta}u_0-{\rm e}^{s\Delta}u_0 + B(u,u)(t)-B(u,u)(s),$$ 
and from the fact that  the heat semigroup is weakly continuous in $\mathcal M(\mathbb R^2)$: ${\rm e}^{t\Delta}u_0\rightharpoonup {\rm e}^{s\Delta}u_0$ as $t\searrow s\ge 0$. 
Now we write 
\be
B(u,u)(t)-B(u,u(s)) = I_1+I_2,\label{sum}
\ee
where 
$$I_1=\int_0^s\nabla\left({\rm e}^{(s-\sigma)\Delta} -{\rm e}^{(t-\sigma)\Delta}\right)\cdot(u(\sigma)Lu(\sigma))\dsi$$ 
and 
$$I_2=-\int_s^t\nabla{\rm e}^{(t-\sigma)\Delta}\cdot(u(\sigma)Lu(\sigma))\dsi.$$ 
Putting $u=z$ in \rf{B01} we get 
\be 
\| u(\sigma)Lu(\sigma)\|_1\le C\sigma^{-1/2}.\label{el}
\ee 
Since the Gauss--Weierstrass kernel satisfies 
$$\big\|\nabla \big(G(\cdot, s-\sigma)-G(\cdot, {t-\sigma})\big)\big\|_1
\le C(t,s)(s-\sigma)^{-1/2}$$ with $C(t,s)\to 0$ as $t\searrow s\ge 0$, by 
$\int_0^t(t-\sigma)^{-1/2}\sigma^{-1/2}\dsi=\pi$ and 
the Lebesgue dominated convergence theorem we arrive at $\|I_1\|_1\to 0$. 

Concerning $I_2$ we note that if $s>0$, then by estimates \rf{el}, \rf{lin2} and $\int_s^t(t-\sigma)^{-1/2}\sigma^{-1/2}\dsi=\int_{\frac{s}{t}}^1(1-\rho)^{-1/2}\rho^{-1/2}\,{\rm d}\rho \to 0$ as $t\searrow s$, we get 
$\|I_2\|_1\to 0$ as $t\searrow s>0$. 
For $s=0$, taking a smooth compactly supported test function $\varphi\in C^1_0(\R^2)$ we see that 
\bea
& &\left| -\int \int_0^t \nabla\cdot({\rm e}^{(t-s)\Delta}(u(s)Lu(s)))\ds\, \varphi(x) \dx\right| \nonumber\\ &=& 
\left|\int_0^t\int {\rm e}^{(t-s)\Delta}(u(s)Lu(s))\cdot\nabla\varphi (x)\dx \ds\right|\nonumber\\ 
&\le& C\int_0^t\sigma^{-1/2}\dsi\ \ \|\nabla\varphi\|_\infty\to 0\nonumber
\eea
as $t\to 0$.
Since such test functions are dense in $C_0(\R^2)$ and $\|B(u,u)(t)\|_1$ are uniformly bounded for $t>0$, the weak convergence $B(u,u)(t)\rightharpoonup 0$ as $t\to 0$ follows. 
\end{proof}

\begin{rem}\label{rem:uniq}
For a fixed $\tau$ and for sufficiently small $\|u_0\|_{{\mathcal M}(\R^2)}$, the solution of \rf{pop}--\rf{ini} constructed in the space ${\mathcal X}$ is unique. 
Indeed, the  existence of the {\em unique} local-in-time solution is proved in the space ${\mathcal X}_T={\mathcal C}_{\rm w}([0,T]; {\mathcal M}(\mathbb R^2))\cap \{u:(0,T)\to L^p(\mathbb R^2)\ : \ \sup_{0<t<T}t^{1-1/p}\|u(t)\|_p<\infty\}$ for sufficiently small $T>0$. For sufficiently  small  $\|u_0\|_{{\mathcal M}(\R^2)}$ and small $T>0$,  the right-hand side of equation \rf{D} defines the contraction on a  suitably chosen ball $\{ u:\sup_{0<t<T}t^{1-1/p}\| u(t)-{\rm e}^{t\Delta}u_0\|_p \le r\}$ so that there is the unique solution $u(t)$ on $[0,T]$ which is, moreover, close to $u_0$ and to ${\rm e}^{t\Delta}u_0$. 
Then, solutions of \rf{D} become regular for all $t>0$, so that their uniqueness on the whole half-line $(0,\infty)$ is guaranteed by a standard argument, see e.g. \cite[Th. 4.5 (ii)]{GMO}. 
\end{rem}


\section*{Acknowledgements}
This work originates from discussions when the first and the third authors have visited Universidad de Chile in December 2012, supported by  FONDECYT grant 1090470, and the second author --- by FONDECYT 1130790. 
The preparation of this paper was also partially supported by  the Foundation for Polish Science operated within the Innovative Economy Operational Programme 2007--2013 funded by European Regional Development Fund (Ph.D. Programme: Mathematical Methods in Natural Sciences) and the NCN grant  2013/09/B/ST1/04412.


\end{document}